\documentclass[12pt]{amsart}
\usepackage{xcolor}
\usepackage[draft]{todonotes} 
\usepackage[english]{babel}
\usepackage{geometry}
\usepackage[T1]{fontenc}
\usepackage[latin1]{inputenc}
\usepackage{amsmath,amssymb,mathrsfs,amsthm, tikz-cd,mathrsfs}
\usepackage{soul}
\usepackage{epsfig}
\usepackage{hyperref}
\usepackage{marvosym}

\usepackage{graphicx}
\usepackage{xypic}
\usepackage{enumitem}
\usepackage{datetime}
\usepackage{xcolor}

\usepackage{hyperref} 
\hypersetup{
    colorlinks=true,
    linkcolor=blue,
    urlcolor=red,
    pdftitle={},
    }

\DeclareMathOperator*\Rprod{%
\mathchoice
  {\ooalign{$\displaystyle\prod$\cr\hidewidth$\displaystyle\coprod$\hidewidth\cr}}%
  {\ooalign{$\textstyle\prod$\cr\hidewidth$\textstyle\coprod$\hidewidth\cr}}%
  {\ooalign{$\scriptstyle\prod$\cr\hidewidth$\scriptstyle\coprod$\hidewidth\cr}}%
  {\ooalign{$\scriptscriptstyle\prod$\cr\hidewidth$%
  \scriptscriptstyle\coprod$\hidewidth\cr}}}

\newcommand\Vbullet{\raisebox{-2.1pt}{\kern-0.4pt\vbox{\baselineskip4pt\lineskiplimit0pt%
\hbox{$\bullet$}\hbox{$\bullet$}}}}

\newtheorem{theorem}{Theorem}[section]
\newtheorem*{theorem*}{Theorem}
\newtheorem{proposition}[theorem]{Proposition}

\newtheorem{corollary}[theorem]{Corollary}

\newcommand{\C}{\mathbb{C}}

\newcommand{\Z}{\mathbb{Z}}

\newcommand{\R}{\mathbb{R}}

\newcommand{\p}{\mathbb{P}}
\newcommand{\eps}{\varepsilon}

\newcommand{\Q}{\mathbb{Q}}
\newcommand{\N}{\mathbb{N}}

\newcommand{\pt}{\partial}

\begin{document}
\title{On the regularized products  of some Dirichlet series}


\author{Mounir Hajli}




\begin{abstract}
In this paper, we show that the regularized determinants of some Dirichlet series are multiplicative. As an application, we give generalizations of  Lerch's formula for the classical gamma function and we determine the sum of some Dirichlet series generalizing Euler's formula on the sum of the reciprocal of squares. We recover 
the results of Kurokawa and Wakayama, and give a new proof for some Euler's formulas.

\end{abstract}

\maketitle

\begin{center}
\textit{MSC\, 2000}: 11M36. 
\end{center}
\tableofcontents

\section{Introduction}

\vskip 1cm

    Let $0<b_1\leq b_2\leq \ldots$ be a sequence of real numbers, and assume that 
   \[\zeta(s)=\sum_{k=1}^\infty b_k^{-s}\]
    converges absolutely on $\mathrm{Re}(s)\gg 0$, and can be continued meromorphically to the whole complex plane and holomorphic at $s=0$. The   regularized 
   determinant of $b_1,b_2,\ldots$ is defined as
   follows
    \[
 \Rprod_{k=1}^\infty  b_k:= \exp(-\zeta'(0)).
  \]   
   When $(b_k)_{k=1,2,\ldots}$ is the spectrum of 
   a positive self-adjoint operator $A$, then using techniques of heat-kernels, one can show  that $\zeta_A(s)=\sum_{k=1}^\infty b_k^{-s}$   can be continued analytically to $s=0$, and $ \Rprod_{k=1}^\infty  b_k$ is called the regularized determinant of the operator $A$ (see \cite{RaySinger}). \\
   
   The zeta regularization  method can be applied to the sequence 
   $(n+x)_{n\in\N}$ where  $x>0$ (Here, $\N=\{0,1,2,\ldots\}$). We obtain the regularized product $ \Rprod_{n=0}^\infty  (n+x)$, which can be  viewed as a function in the variable $x$, and it can be continued analytically to $\C$.   
  The classical Gamma function  $\Gamma(x) $ can be defined from 
   the Hurwitz zeta function via the Lerch formula \cite{Lerch},
   \begin{equation}\label{lerch}
 \Rprod_{k=0}^\infty  (k+x):=\exp(-\frac{\pt \zeta_H}{\pt s}(x,0))=\frac{\sqrt{2\pi}}{\Gamma(x)}\quad \text{for all  $x\in \R$},
   \end{equation}
   where $\zeta_H(s,x)=\sum_{k=0}^\infty (k+x)^{-s}$ for $x>0$ (the Hurwitz zeta function). \\
   
    Only a few general methods  for the exact evaluation of $  \Rprod_{k=1}^\infty  b_k$ are available.  Many authors aimed to give complete results for a number of cases where it was possible to find formulas for the zeta invariants in terms of some known  special functions, see for instance \cite{Bordag1,Bordag2,Kurokawa3, Kurokawa2,  Voros}.   In \cite{Bordag1,Bordag2},  the authors determine the regularized product of the spectrum  of the Laplacian on the $D$-dimensional ball.\\
    
    In the case of the complex projective space $\p^n(\C)$ endowed with its Fubini-Study metric, the computation of  its holomorphic analytic torsion (see \cite{RaySinger}), which is  by definition  a linear combination of the  regularized determinants of the Laplacians $\Delta^q$ acting on  the space of smooth $(0,q)$-forms on $\p^n(\C)$ for $q=0,\ldots,n$ , was an important step toward the formulation of an arithmetic Riemann-Roch theorem in the context of Arakelov Geometry (see \cite{GSZ} for the computation of the holomorphic analytic torsion of $\p^n(\C)$, and \cite[Chapter VIII]{Soule} for the general formulation of the arithmetic Riemann-Roch theorem).  In \cite[Section 2]{HajliJNT20}, we gave a new method for the computation of the regularized determinant of $\Delta^0$ on  $\p^n(\C)$ endowed with the Fubini-Study metric.  This approach can be adapted easily to get the regularized determinant of $\Delta^q$ for $q=1,\ldots,n$.     One can use the arithmetic Riemann-Roch theorem to compute the analytic torsion of the  line bundles $\mathcal{O}(m)$, with $m\geq 0$, on $\p^n(\C)$ endowed with their Fubini-Study metrics, this was done by L. Weng in \cite{Weng}.  In \cite{HajliJMPA}, we defined a new class of singular Laplacians  $\Delta_{\overline{\mathcal{O}(m)}_\infty}$ associated with  $\mathcal{O}(m)$ on  $\p^1(\C)$ endowed with their canonical metrics. We showed that the spectrum of  $\Delta_{\overline{\mathcal{O}(m)}_\infty}$ can be determined explicitly in terms of Bessel functions (see \cite[Theorem 1.4]{HajliJMPA}). In \cite{HajliTorJNT}, we showed that  the  zeta function $\zeta_{\Delta_{\overline{\mathcal{O}(m)}_\infty}}(s)$ associated with the spectrum of $\Delta_{\overline{\mathcal{O}(m)}_\infty}$ can be continued analytically to $s=0$, and using \cite[Theorem 3.6]{HajliTorJNT} we obtained  that \[\zeta_{\Delta_{\overline{\mathcal{O}(m)}_\infty}}(0)=-\frac{2}{3}-\frac{m}{2}\footnote{Surprisingly, this value coincides with the value of the  zeta function of $\Delta_{\overline{\mathcal{O}(m)}}(s)$ at $s=0$, when the metrics are smooth; Indeed, we can show that , in the case of smooth metrics, this is a topological invariant.},\] and, we computed  the regularized determinant  \[\Rprod_{\lambda\in \mathrm{Spec}(\Delta_{\overline{\mathcal{O}(m)}_\infty}) \setminus\{0\}} \lambda
    =\frac{(m+2)^{m+1}}{((m+1)!)^2}\exp\left(-4\zeta_\Q'(-1)+\frac{1}{6}\right), \] (see \cite[Theorem 1.1]{HajliTorJNT}). In \cite{HajliJMAA}, we developed a generalized spectral theory explaining these results.  \\



Let $0<c_1\leq c_2\leq \ldots$ and $0<b_1\leq b_2\leq \ldots$ be two sequences of  positive real numbers. We assume that $\sum_{k=1}^\infty b_k^{-s}$  and   $\sum_{k=1}^\infty c_k^{-s}$  converge for $\mathrm{Re}(s)\gg 0$, 
and admit  meromorphic continuations to $\C$ which are holomorphic at $s=0$.  An interesting question is to evaluate the following ratio:
\begin{equation}\label{ratio}
\frac{\Rprod_{k=1}^\infty(b_kc_k)}{\Rprod_{k=1}^\infty b_k  \Rprod_{k=1}^\infty c_k}=?
\end{equation}
This question was raised in \cite[p. 100]{Soule}. In \cite{HajliQJM}, we give a partial
answer to this question.  Garate and Friedman in
 \cite{FriedmanDiscrepancies}
develop a different method which allows them to compute some special cases. \\

 In this paper, we use the main result of \cite{HajliQJM} to evaluate the following regularized products
\begin{equation}\label{product}
 \Rprod_{k=0}^\infty ((k+x)^m-\eps y^m),
\end{equation}
where  $m$ is a  positive integer and $\eps^2=1$.  
Our  main result is,
\begin{theorem}\label{gammam}
Let $m$  be a positive integer $\geq 2$.  Let $L(x):=\Rprod_{k=0}^\infty(k+x)$, with $x\in \C$. 
We have 
 \[
 \Rprod_{k=0}^\infty \left( (k+x)^m-\eps y^m \right)=\prod_{\xi^m=1} L(x-\xi \eps^{\frac{1}{m}} y)\quad \forall x,y\in \C,
\]

where $\eps\in \R$ such that $\eps^2=1$, $\eps^{\frac{1}{m}}$ is an 
$m$-th root of $\eps$. 

\end{theorem}
In this situation,
we see  that the regularized determinant is multiplicative, i.e. \eqref{ratio} is equal to $1$. In general, 
\eqref{ratio} may be $\neq 1$, see \cite[Eq. 11]{HajliQJM}. \\

Let us review some basic facts about the classical Gamma function. Let $\Gamma(z)$ be the classical Gamma function. This  function can be defined as follows
\begin{equation}
\Gamma(z)=\underset{n\rightarrow \infty}{\lim} \frac{n! n^z}{z(z+1)\cdots (z+n)}\quad\text{for}\;z\in \C\setminus \Z.
\end{equation}
Bohr and Mollerup used  this definition to give a  characterization for the Gamma function (see \cite[p.14]{ArtinGamma} for more details).  We know that

\begin{equation}\label{Gamma1}
\Gamma(z)=\int_0^\infty t^{z-1}e^{-t}dt\quad \text{for}\;z\in \C\setminus \Z.
\end{equation}
The Gamma function is a meromorphic function in the complex plane, with simple poles in $\Z_{\leq 0}$. Weierstrass theory shows that
\begin{equation}\label{weierstrass}
\frac{1}{\Gamma(s)}=s e^{\gamma s}\prod_{n=1}^\infty \left(1+\frac{s}{n}\right) e^{-\frac{s}{n}},
\end{equation}
where $\gamma$ is the Euler-Mascheroni constant. 
Using the Lerch formula, which is proved in Appendix \ref{appendice}, our theorem recovers the result of  Kurokawa and Wakayama  \cite{Kurokawa1}.  Our approach is new and self-contained.   Kurokawa and Wakayama   use of  the infinite product expression  \eqref{weierstrass} for the Gamma function and the Lerch's formula \eqref{lerch} which   is a crucial step in their proof. 
\\

As an application of Theorem \ref{gammam}, we obtain the following  identities: 
\[
\frac{1}{2\pi}\sum_{k\in \Z} \frac{2y}{k^2+y^2}=\coth(\pi y),
\]

\[
\frac{1 }{\pi \sqrt{2}}\sum_{k\in \Z}^\infty\frac{4y^3}{k^4+y^4}=\frac{\sinh(\sqrt{2} \pi y)+\sin(\sqrt{2} \pi y) }{\cosh(\sqrt{2} \pi y)-\cos(\sqrt{2} \pi y)},
\]
and
\begin{equation}\label{y2nk}
\frac{n}{ \pi} \sum_{k\in \Z}\frac{y^{2n-1}}{k^{2n}+y^{2n}}=\sum_{l=0}^{n-1} e^{\frac{\pi i l}{n}} \cot(\pi e^{\frac{\pi i l}{n}}y) \;\text{ for  $n\geq 3$,
}
\end{equation}
for any $y\in \R\setminus \{0\}$
 (see Theorem \ref{exampleseries}). From the first identity, we obtain a new proof for Euler's theorem: 
\[
\zeta_\Q(2j)=\frac{(-1)^{j+1} 2^{2j-1}}{(2j)!} \pi^{2j}B_{2j}\quad j=1,2,\ldots
\]
By considering the Laurent expansion near $y=0$ of the third identity, we obtain a second proof for  Euler's theorem. \\

  
  Motivated by the above results, we established the following theorem which generalizes identity \eqref{y2nk}.
  \begin{corollary}(see Theorem \ref{general})
{Let  $x,y\in \C$ with $\mathrm{Re}(x)> -1$ and $\mathrm{Re}(y) >0$}. For any positive integer $m\geq 2$,  we have

\begin{equation}\label{genEuler}
\sum_{k=1 }^\infty \frac{1}{(k+x)^m+y}= -\frac{1}{m}\sum_{\xi^m=1} \xi (-1)^{\frac{1}{m}} y^{\frac{1}{m}-1}
\Psi(x-\xi (-1)^{\frac{1}{m}} y^{\frac{1}{m}}+1),
\end{equation}
where $(-1)^{\frac{1}{m}}$ denotes an $m$-th root of $-1$ and $\Psi$ is the logarithmic derivative of the Gamma function.

\end{corollary}

Note that the left hand side of \eqref{genEuler} is   Mellin's transform  of 
  \[
  \theta_m(t,x;y)=\sum_{k=1}^\infty e^{-((k+x)^m+ y)t},
  \]
 evaluated  at $s=1$.
  This function is studied Section \ref{sec3}. 
  When $m=2$ and $x=y=0$, this is    the classical Jacobi theta series
  \[
  \sum_{k=1}^\infty e^{-k^2t }.
  \]
It follows from the Poisson summation formula that
$
\theta_2(t,0,0)=\tfrac{\sqrt{\pi}}{\sqrt{t}} \theta_2(\tfrac{1}{t},0,0)-\frac{1}{2}.
$
Hence, we obtain \[\theta_2(t,0,0)=\tfrac{\sqrt{\pi}}{2\sqrt{t}} -\tfrac{1}{2}+O(\sqrt{t})\quad (t\rightarrow 0^+). \]

\section{On the regularized determinant of some Dirichlet series}

\vskip 1cm

Let $0<b_1\leq b_2\leq \ldots$ be an unbounded and nondecreasing  sequence of positive real numbers. To this data, we associate the following series 
\[
\theta(t):=\sum_{k=1}^\infty e^{-b_kt} 
\quad \text{ for $t>0$ (whenever it converges)}.
 \] 
We suppose that $\theta$ satisfies the following three conditions:

 \begin{enumerate}
 \item[(\textbf{$\Theta $}1):] The series  
 $\sum_k e^{-b_kt}$  converges for any $t>0$,

\item[(\textbf{$\Theta $}2):] The series $\theta(t)$ admits an asymptotic  expansion for $t\rightarrow 0^+$
\[
\theta(t)\sim\sum_{m\in \N} c_{i_m} t^{i_m}, 
\]
where $(i_n)_{n\in \N}$ is an unbounded and non-decreasing sequence of real numbers. Moreover, we assume that $c_{i_0}\neq 0$,


\item[(\textbf{$\Theta $}3):]  $\sum_k e^{-b_kt}=O(e^{-\kappa t})$ for $t\rightarrow +\infty$, for some positive real number $\kappa$.  
\\

\end{enumerate}

We set \begin{equation}\label{itheta}
i_0:=i(\theta),
\end{equation} and call it the \textit{index} of $\theta$. The zeta function $\xi$ associated with $\theta$, is by definition the function 
\[
\xi(s):=\frac{1}{\Gamma(s)}\int_0^\infty t^{s-1}\theta(t)dt=\sum_{k=1}^\infty 
\frac{1}{b_k^s} \quad \text{for } \mathrm{Re}(s)>-i(\theta).
\]

The properties $(\Theta1),(\Theta2)$ and $(\Theta 3)$ imply that $\xi$ can be continued into a meromorphic function on the whole complex plane, which is holomorphic at $s=0$.\\

In the sequel, we consider the following sequence 
$(Q(k))_{k\in \N}$, where  $Q$ is a monic  polynomial of degree $\ell$ having  only  distinct roots with positive real part, and $Q(k)>0$ for any $k=1,2,\ldots$  It is known that the Dirichlet  series
\[
\sum_{k=1}^\infty \frac{1}{Q(k)^s},
\]
converges on any compact subset of $\mathrm{Re}(s)>1/\ell$, and has a meromorphic continuation to the whole complex plane (see \cite[Theorem 3.2]{HajliQJM}). The regularized product of the sequence $(Q(k))_{k=1,2,\ldots}$ is by definition 
\[
 \Rprod_{k=1}^\infty  Q(k)=\exp\left(-\frac{\pt}{\pt s} \left(\sum_{k=1}^\infty \frac{1}{Q(k)^s}\right)_{|_{s=0}}\right).
\]

We set
\[
L(x):= \Rprod_{n=0}^\infty  (n+x)\quad \forall x\in \R.
\]
(see Appendix \ref{appendice} for the properties of this function).

We have
\begin{theorem}\label{regprod} Under the conditions above,
\[
 \Rprod_{k=0}^\infty Q(k)=\prod_{i=1}^{\ell} L(d_i)= \frac{(2\pi)^{\ell/2}}{ \prod_{i=1}^{\ell}\Gamma(d_i)},\]
where $d_1,\ldots,d_\ell$ are the roots of $Q$.
\end{theorem}

\begin{proof}

Let $Z(s)=\sum_{k=1}^\infty \frac{1}{Q(k)^s}$. By definition,  
\[
 \Rprod_{k=1}^\infty  Q( k) :=\exp(- \frac{\pt}{\pt s} Z(s)_{|_{s=0}}).\]
Let 
 \[
 g_i(s)=\sum_{k=1}^\infty \frac{1}{(k+d_i)^s}\quad  \text{with $\mathrm{Re}(s)>1$, for $i=1,\ldots,l-1,\ell$}.
 \]
We put $f=g_\ell$. We have 
\[
f(s)=\frac{1}{\Gamma(s)}\int_0^\infty t^{s-1}\theta(t) dt ,
\]
where $\theta(t)=\sum_{k=1}^\infty e^{-(k+d_\ell)t}$ for $t>0$. We have 
\[
\theta(t)=\frac{e^{-(1+d_\ell)t}}{1-e^{-t}}=\frac{1}{t}+\text{(an analytic series) (for $0<t\ll 1$)} .
\]  Then, it is easy to see that $i(\theta)=-1$. \\

By \cite[Theorem 3.4]{HajliQJM},  there exists a polynomial $F$ in $\ell$ variables such that
\begin{equation}\label{zgi}
\exp(-Z'(0))=\prod_{i=1}^{\ell} \exp(-g_i'(0)) \exp(F(d_1,\ldots,d_\ell)),
\end{equation}
(see also \cite[eq. (11)]{HajliQJM}).

But $i(\theta)=-1$, then
\[
F(d_1,\ldots,d_\ell)=0.
\]
On the other hand,
\[
g_i(s)=\sum_{k=1}^\infty \frac{1}{(k+d_i)^s}=\zeta_H(s;d_i)-
\frac{1}{d_i^s}\quad\text{for $i=1,\ldots,l$},
\]
So,
\[
\exp(-g_i'(0))=L(d_i+1)\quad\text{for $i=1,\ldots,l$}
\]

By Lerch's formula  \cite{Lerch} (see Appendix \ref{appendice} for a  different proof), we get
\[
\exp(-g_i'(0))=\frac{\sqrt{2\pi}}{d_i\Gamma(d_i)}\quad\text{for $i=1,\ldots,l$}
\]

Gathering all these computations, \eqref{zgi} becomes 
\[
 \Rprod_{k=1}^\infty  Q(k)= \frac{(2\pi)^{\deg Q/2}}{ \prod_{i=1}^{\ell} d_i\Gamma(d_i)}.  
\]
So,
\[
 \Rprod_{k=0}^\infty Q(k)= \frac{(2\pi)^{\deg Q/2}}{ \prod_{i=1}^{\ell}\Gamma(d_i)}.
\]
\end{proof}



\begin{proof}[Proof of Theorem \ref{gammam}]
Let $m$ be a positive integer.   We consider the polynomial in $t$:
\[
Q(t)=(t^n+x)^m-\eps y^m,
\]
with $\eps^2=1$, and $x,y\in \C$. We choose an $m$-th root of
$\eps$, which we denote by $\eps^{\frac{1}{m}}$. Note that 
\[
Q(t)=\prod_{\xi^m=1} \left( t+x -\xi \eps^{\frac{1}{m}} y\right).
\]
We conclude the proof by using Theorem \ref{regprod}. 
\end{proof}

Using Lerch's formula, we obtain 
\begin{corollary} With the same assumptions as in Theorem \ref{gammam}, we have 
 \[
 \Rprod_{k=0}^\infty \left( (k+x)^m-\eps y^m \right)=
\frac{(2\pi)^{m/2}}{ \prod_{\xi^m=1} \Gamma(x-\xi \eps^{\frac{1}{m}} y)}\tag{Lerch-Kurokawa-Wakayama's formulas}.
\]
\end{corollary}

\section{On the variation of  the regularized product  $\Rprod_{k=0}^\infty \left( (k+x)^m+ y\right)$}\label{sec3}

For $m=1,2,\ldots$, $x\geq 0$ and $y\in \C$, we set
\begin{equation}\label{asymptotic}
\theta_m(t,x;y)=\sum_{k=1}^\infty e^{-((k+x)^m+ y)t}\quad \text{for}\;t>0.
\end{equation}
The following proposition is important in our study of the variation of  \[ \Rprod_{k=0}^\infty ((k+x)^m+y)\] as a function in the variable $y$.

\begin{proposition}\label{thetam}

The function $t\mapsto \theta_m(t,x;y)  $ is smooth  on $(0,\infty)$ and has an 
asymptotic expansion for small  positive $t$ of the form
\[
\theta_m(t,x;y)\sim   \frac{1}{\Gamma(\frac{1}{m}+1)} t^{-\frac{1}{m}}+\sum_{k=0}^\infty c_k(x;y) t^k,\]
(where $c_k(x;y)$ are constants which depend on $x$ and $y$), and which decays exponentially at infinity, more precisely, we have
for $t$ sufficiently large, $
|\theta_m(t,x;y)| \leq C(y)e^{-t}$ for $t\gg 1$, where  $C(y)$ is a constant  which depends continuously on $y$). 

Moreover, for $y\geq 0$ and $x\geq 0$,
\[ \zeta_m(s,x;y):=\frac{1}{\Gamma(s)}\int_0^\infty t^{s-1}\theta_m(t,x;y)dt=\sum_{k=1}^\infty  \frac{1}{((k+x)^m+y)^s}\quad \text{for}\;\mathrm{Re}(s)>\frac{1}{m},\]
which admits a meromorphic continuation to $\C$ and has only one simple pole at
$s=\frac{1}{m}$. In particular,
\[
i(\theta_m(t,x;y))=-\frac{1}{m}.
\]

\end{proposition}

\begin{proof}

It is clear that $t\mapsto \theta_m(t,x;y)$ is a smooth function.  

Let $\sigma_1<\sigma_2$ be two real numbers. On the closed strip $\sigma_1\leq \mathrm{Re}(s)\leq \sigma_2$, we have for any $r>0$,
\[
\Gamma(s)=O(|s|^{-r}),
\]
as $|s|\rightarrow \infty$ in the strip. This decay follows from the complex version of Stirling's formula  (see \cite[p. 257]{Table2}). On the other hand, it is known that for $\sigma<1$, one has 
\[
\zeta_H(\sigma+it,x)=O(|t|^{1-\sigma})\quad \text{for\;} t\in \R.
\] 
Since  $\theta_m(t,x;0)$ is the inverse Mellin  transform  of 
$\Gamma(s)\zeta_H(ms,,x+1)$, and knowing that 
$\zeta_H(ms,x+1)$ can be continued analytically to $\mathrm{Re}(s)<1/m$, we can use the Mellin inversion theorem to show that $\theta_m(t,x;0)$ has an asymptotic expansion for $t$ small enough, 
\[
\theta_m(t,x;0)\sim   \frac{1}{\Gamma(\frac{1}{m}+1)} t^{-\frac{1}{m}}+\sum_{k=0}^\infty c_k(x) t^k\quad \text{for}\; 0<t\ll 1,
\]
where the dominant term corresponds to the unique pole of 
$\zeta_H(ms,x+1)$. In fact,  the coefficients of the asymptotic expansion can be determined  in terms of the special values of $\zeta_H(ms,x+1)$ 
at the non-negative integers (see \cite[Proposition 2.1]{HajliQJM}). Since 
$\theta_m(t,x;y)=\theta_m(t,x;0)e^{-yt}$, we  can deduce the asymptotic expansion \eqref{asymptotic}. That is  the existence of a sequence of constants $(c_k(x;y))_{k\in \N}$ such that
\begin{equation}\label{ckxy}
\sum_{k=1}^\infty e^{-((k+x)^m+y)t}\sim   \frac{1}{\Gamma(\frac{1}{m}+1)} t^{-\frac{1}{m}}+\sum_{k=0}^\infty c_k(x;y) t^k.
\end{equation}

We have, for $x\geq 0$ and $y\in \C$
\[
\left|\sum_{k=1}^\infty e^{-(k+x)^mt-yt} \right|\leq  \sum_{k=1}^\infty e^{-kt} e^{-\mathrm{Re}(y)t}\quad\forall t>0.
\]
So, we deduce the existence of  a constant $C(y)$ which depends continuously on $y$ such that
\begin{equation}\label{cy}
\left|\sum_{k=1}^\infty e^{-(k+x)^mt-yt} \right|\leq  C(y)e^{-t}\quad\forall t\geq 1.
\end{equation}

 
From \eqref{ckxy} and  \eqref{cy}, and using \cite[Proposition 2.1]{HajliQJM}, we conclude that  $\zeta_m(s,x;y)$
admits a meromorphic continuation to $\C$ with only one  pole at
$s=\frac{1}{m}$.  The asymptotic expansion shows that this pole is simple.

\end{proof}

\begin{theorem}\label{general}
Let $x,y >0$. For any $m=2,3,\ldots,$ we have  

\[
\sum_{\xi^m=1} \Gamma(x-\xi (-1)^{\frac{1}{m}} y^{\frac{1}{m}} +1)\frac{\pt}{\pt y}\left(\frac{1}{ \prod_{\xi^m=1} \Gamma(x-\xi (-1)^{\frac{1}{m}} y^{\frac{1}{m}} +1)}\right)=\sum_{k=1}^\infty \frac{1}{(k+x)^m +y},
\]
where $(-1)^{\frac{1}{m}}$ denotes an $m$-th root of $-1$.
\end{theorem}

\begin{proof}

By Theorem \ref{gammam}, it is enough to prove the following identity
\begin{equation}\label{der}
\frac{\pt}{\pt y} \log \Rprod_{k=1}^\infty  \left( (k^\ell+x)^m + y\right)=\sum_{k=1}^\infty \frac{1}{((k+x)^m + y)},
\end{equation}
for $x>0$ and $y\in\C$ sufficiently small,  and to notice that the left hand side of this equality is 
$-\frac{\pt \zeta_m}{\pt s}(s,x;-\eps y)_{|_{s=0}}$ (see the notations in Proposition \ref{thetam}). \\

  We have
\[
\begin{split}
\frac{\pt}{\pt y}\zeta_m(s,x;y)=&-\frac{1}{\Gamma(s)}\int_0^\infty t^s \theta_m(t,x;y)dt\\
=&\frac{s}{\Gamma(s)}\int_0^\infty
t^{s-1}\Theta_m(t,x;y)dt\quad \forall y\in \C,
\end{split}
\]
where $
\Theta_m(t,x;y):=\int_t^\infty \sum_{k=1}^\infty e^{-( (k+x)^m+y)u}du$, which decays exponentially as $t\rightarrow \infty$.\\

We set \[
g(t)=\theta_m(t,x;y)-\Gamma(1+\frac{1}{m})\frac{1}{t^{\frac{1}{m}}}\quad \forall t>0.
\] By Proposition \ref{thetam}, $g$  is  a bounded function for $t$ small enough.  Using the asymptotic expansion of
$\theta_m(t,x;y),$ we  have
\[
\begin{split}
\zeta_m(s,x;y)=&\frac{1}{\Gamma(s)} \Gamma(1+\frac{1}{m}) \int_0^1 t^{s+\frac{1}{m}-1}dt +\frac{1}{\Gamma(s)} \int_0^1 t^{s-1} g(t) dt \\
&
 +\frac{1}{\Gamma(s)} \int_1^\infty t^{s-1}\theta(t) dt.
\end{split}
\]
It follows that $\zeta_m(s,x;y)$ can be continued into a holomorphic function in the variables  $s$ and $y$, for any $s$ in an open neighborhood of $0$  and $y$ in any  bounded subset of $\C$. \\

In the sequel, we shall use the same notation $\zeta_m(s,x;y)$ to denote its holomorphic 
continuation. We have
\[
\frac{\pt^2}{\pt y \pt s} \zeta_m(s,x;y)=\frac{\pt^2}{\pt s \pt y} \zeta_m(s,x;y),
\]
which holds on an open neighborhood of $s=0$. 

We obtain

\[
\frac{\pt^2 \zeta_m}{\pt y\pt s} \zeta_m(0,x;y)=\frac{\pt}{\pt s} \left(\frac{s}{\Gamma(s)}\int_0^\infty
t^{s-1}\Theta_m(t,x;y)dt\right)_{|_{s=0}}. \]

But,
\[
\Theta_m(t,x,;y)=\sum_{k=1}^\infty  \frac{1}{((k+x)^m+y)}e^{-((k+x)^m+y)t}\quad \text{for}\; t>0,
\]
So,
\[
\begin{split}
\frac{\pt}{\pt s} \left(\frac{s}{\Gamma(s)}\int_0^\infty
t^{s-1}\Theta_m(t,x;y)dt\right)_{|_{s=0}}=& \frac{\pt}{\pt s}\left(s\sum_{k=1}^\infty \frac{1}{((k+x)^m+y)^{s+1}} \right)_{|_{s=0}} \\
=& \frac{\pt}{\pt s} (s\zeta_m(s+1,x;y))_{|_{s=0}}. 
\end{split}
\]

Since 
 $\zeta_m(s,x;y)$ is holomorphic at
$s=1$ (use Proposition \ref{thetam}). Then \[
\frac{\pt}{\pt s} (s\zeta_m(s+1,x;y))_{|_{s=0}}=\zeta_m(1,x;y).\] 
Hence
\[
\frac{\pt^2 \zeta_m}{\pt y\pt s} (0;x,y)=\zeta_m(1,x;y). 
\]
Therefore,
\[
\frac{\pt}{\pt y} \exp\left(-\frac{\pt}{\pt s} \zeta_m(s,x;y)_{|_{s=0}}\right)=
-\frac{\pt}{\pt s} \zeta_m(s,x;y)_{|_{s=0}} \zeta_m(1,x;y).
\]
This concludes the proof of the theorem.

\end{proof}


\section{Some applications}

We have
\[
 \Rprod_{k=1}^\infty  (k^4+y)=2y^{-\frac{1}{2}} (\cosh(\sqrt{2}\pi y^{\frac{1}{4}})-\cos(\sqrt{2}\pi y^{\frac{1}{4}}))\quad \forall y>0,
\]
(see  \cite[p. 943]{Kurokawa1}) which can be proved using  the reflection formula for the Gamma function.

\begin{theorem}\label{exampleseries}
We have
\begin{enumerate}
\item

\[
\frac{1}{2\pi}\sum_{k\in \Z} \frac{2y}{k^2+y^2}=\coth(\pi y),
\]

\item

\[
\frac{1 }{\pi \sqrt{2}}\sum_{k\in \Z}^\infty\frac{4y^3}{k^4+y^4}=\frac{\sinh(\sqrt{2} \pi y)+\sin(\sqrt{2} \pi y) }{\cosh(\sqrt{2} \pi y)-\cos(\sqrt{2} \pi y)},
\]

\item For any $n\geq 3$,
\[
\frac{n}{ \pi} \sum_{k\in \Z}\frac{y^{2n-1}}{k^{2n}+y^{2n}}=\sum_{l=0}^{n-1} e^{\frac{\pi i l}{n}} \tan(\pi e^{\frac{\pi i l}{k}}y),
\]

\end{enumerate}
for any  $y\in \R\setminus\{0\}$.
\end{theorem}

\begin{proof}
\item
\begin{enumerate}
\item   For any $y>0$, we have
\[
 \Rprod_{k=0}^\infty (k^2+y)= 2y^{\frac{1}{2}} \sinh(\pi y^{\frac{1}{2}}).
\]
(this is identity is a direct consequence of  (3) of Theorem \ref{lerchthm}).  So, from 
\eqref{der}, we get the following identity
\begin{equation}\label{seriesk2y2}
\sum_{k=1}^\infty \frac{1}{k^2+y}=-\frac{1}{2y}+\frac{\pi }{2 y^{\frac{1}{2}}} \coth(\pi y^{\frac{1}{2}}). 
\end{equation}


That is
\[
\frac{1}{2\pi }\sum_{k\in \Z}^\infty \frac{2y}{k^2+y^2}=\coth(\pi y). 
\]

\item

From \cite{Kurokawa1}, we know that
\[
\frac{(2\pi)^2}{\underset{\zeta^4=1}\prod \Gamma(-\zeta \frac{1+i}{\sqrt{2}}y )}=2y^2(\cosh(\sqrt{2} \pi y)-\cos(\sqrt{2} \pi y)),
\]
(The authors uses the reflection formula  to prove this equation. See \ref{reflection} for a proof of this formula).  
Using again \eqref{der}, we get for any $y>0$,
\[
\left( y^{-\frac{1}{2}} \bigl(\cosh(\sqrt{2}\pi y^{\frac{1}{4}})-\cos(\sqrt{2}\pi y^{\frac{1}{4}})\bigr)\right)^{-1} \frac{d}{dy} \left(y^{-\frac{1}{2}} \bigl(\cosh(\sqrt{2}\pi y^{\frac{1}{4}})-\cos(\sqrt{2}\pi y^{\frac{1}{4}})\bigr)\right)=\sum_{k=1}^\infty\frac{1}{k^4+y}.
\]
That is,
\[ \sum_{k=1}^\infty\frac{1}{k^4+y}=-\frac{1}{2y}+\frac{\sqrt{2}\pi}{4y^{\frac{3}{4}} } \left( \frac{\sinh(\sqrt{2} \pi y^{\frac{1}{4}})+\sin(\sqrt{2} \pi y^{\frac{1}{4}}) }{\cosh(\sqrt{2} \pi y^{\frac{1}{4}})-\cos(\sqrt{2} \pi y^{\frac{1}{4}})} \right).\]

Then
\[
\frac{4}{\pi \sqrt{2}}\sum_{k\in \Z}\frac{y^3}{k^4+y^4}=\left( \frac{\sinh(\sqrt{2} \pi y)+\sin(\sqrt{2} \pi y) }{\cosh(\sqrt{2} \pi y)-\cos(\sqrt{2} \pi y)} \right).
\]
\item 

Using the reflection formula, we obtain
\[
\frac{(2\pi)^{2n}}{\prod_{\xi^{2n}=-1}\Gamma(\xi z)}=(-1)^k 2^k z^k e^{\frac{\pi i}{2}n(n-1)} \prod_{l=0}^{n-1} \sin(\pi e^{\frac{\pi il}{n} } e^{\frac{\pi i}{2k}}z).
\]
Then,
\[
\Rprod_{k=1}^\infty(k^{2n}+z^{2k})=(-1)^k 2^k z^{-k} e^{\frac{\pi i}{2}n(n-1)} \prod_{l=0}^{n-1} \sin(\pi e^{\frac{\pi il}{n} } e^{\frac{\pi i}{2k}}z).
\]
Taking the logarithmic derivative, we get

\[
\frac{n e^{-\frac{\pi i}{2n}}}{ \pi} \sum_{k\in \Z}\frac{z^{2n-1}}{k^{2n}+z^{2n}}=\sum_{l=0}^{n-1} e^{\frac{\pi i l}{n}} \cot(\pi e^{\frac{\pi i l}{k}} e^{\frac{\pi i}{2n}}z).
\]
By letting $y=e^{\frac{\pi i}{2n}}z$, we obtain the desired identity.
\end{enumerate}
\end{proof}


\begin{corollary}
For any $j=1,2,\ldots$, 
\[
\zeta_\Q(2j)=\frac{(-1)^{j+1} 2^{2j-1}}{(2j)!} \pi^{2j}B_{2j},
\]
where $B_{2j}$ is the $2j$-th Bernoulli number.
\end{corollary}

\begin{proof} Recall that
\[
\begin{split}
\frac{1}{2y}+\frac{\pi}{2y^{\frac{1}{2}}}\coth(\pi y^{\frac{1}{2}})=\sum_{j=1}^\infty \frac{B_{2j}}{(2j)!} (2\pi)^{2j} y^j,
\end{split}
\]
where $B_{2j}$ is the $2j$-th Bernoulli number. From \eqref{seriesk2y2}, it is easy to see that
\[
\sum_{k=1}^\infty \frac{1}{k^{2j}}= \frac{(-1)^{j+1} 2^{2j-1}}{(2j)!} \pi^{2j}B_{2j}.
\]
\end{proof}

\begin{corollary}\label{cor1.2}

We have
\begin{enumerate}[label=(\roman*)]

\item 
\begin{equation}\label{k2y}
\sum_{k\in \Z} \frac{1}{k^2+y^2}=\frac{\pi}{y} \left( \frac{e^{2\pi y} +1}{e^{2\pi y}-1}\right),
\end{equation}

\item 
\begin{equation}\label{Euler}
\zeta_{\mathbb{Q}}(2k)=\sum_{\ell=1}^\infty \frac{1}{\ell^{2k}}=(-1)^{k-1} \frac{(2\pi)^{2k} B_{2k}}{2(2k)! }\tag{Euler},
\end{equation}

\item
$$\sum_{k\in \Z}^\infty\frac{1}{k^4+y^4}= \frac{\pi \sqrt{2}}{4y^3} \left(\frac{\sinh(\sqrt{2} \pi y)+\sin(\sqrt{2} \pi y) }{\cosh(\sqrt{2} \pi y)-\cos(\sqrt{2} \pi y)}\right). $$

\item
$$\sum_{k=0 }^\infty \frac{1}{(k+x)^m}=\zeta_H(m,x)=(-1)^{m-1} \frac{ \Psi^{(m-1)}(x) }{(m-1)!},$$
where $\zeta_H$ is Hurwitz's zeta function.


\end{enumerate}

\end{corollary}

\section{Appendix : On  Lerch's formula}\label{appendice}

The reflection formula was first proved by Euler. In order to obtain his formula, Euler has to prove the following product formula for the sine function: 
\[
\sin(\pi x)=\pi x  \prod_{k=1}^\infty  \left(1-\frac{x^2}{k^2} \right),
\]
which can be deduced from the following identity
\[
\pi \cot(\pi x)=\sum_{n\in \Z}^{e} \frac{1}{n+x}=\frac{1}{x}+2 x\sum_{n=1}^\infty \frac{1}{x^2-n^2}\quad \forall x\in \R\setminus\Z,
\]
where $\overset{e}{ \sum}$ stands for the Eisenstein summation, see \cite{Weil} for the proof of this identity.
Now,  
\[
\Gamma(x+1)= \prod_{k=1}^\infty  \frac{k^{1-x} (k+1)^x}{x+k}.
\]
Therefore, replacing $x$ by $-x$: 
\[
\Gamma(1-x)= \prod_{k=1}^\infty  \frac{k^{1+x} (k+1)^{-x}}{k-x}.
\]
Therefore,
\[
\Gamma(1+x)\Gamma(1-x)= \prod_{k=1}^\infty  \frac{k^2}{k^2-x^2}= \prod_{k=1}^\infty \frac{1}{1-\frac{x^2}{k^2}}.
\]
So,
\[
\Gamma(x+1)\Gamma(1-x)=\frac{\pi x}{\sin(\pi x)}\tag{Reflection formula }.
\]

In the sequel, we give a different proof of  Lerch's formula and the reflection formula for the Gamma function.  We do not claim that our  proof is new. Probably this method is known, but we could not find a reference containing this proof. \\

We  set
\[
L(x)= \Rprod_{n=0}^\infty  (n+x)\; \text{for}\; x\geq 0.
\]

\begin{theorem}\label{lerchthm}
The function $\ell$ extends to an analytic function on $\C$ which possesses the following 
properties:
\begin{enumerate}
\item $L(x)>0$ for any $x>0$,

\item $L(x+1)=\frac{1}{x}L(x)$ for any $x\in \R$,

\item  For $x\in \C$, \[
L(x)L(-x)=-2x\sin(\pi x),
\]
\item $L(1)=\Rprod_{n=1}^\infty n=\sqrt{2\pi}$.
 
\end{enumerate}
Moreover,
\[
L(x)=\frac{(2\pi)^\frac{1}{2}}{\Gamma(x)}\quad\forall x\in \C \tag{Lerch's formula}, 
\]
and
\begin{equation}\label{reflection}
\frac{1}{\Gamma(x)\Gamma(-x)}=-\frac{1}{\pi} x\sin(\pi x)\quad \forall x\in \C\tag{Reflection formula}.
\end{equation}

\end{theorem}

\begin{proof}

Let $x>0$. $\zeta_H(s,x)$ has the following integral representation for any $\mathrm{Re}(s)>1$,
\[
\begin{split}
\zeta_H(s,x)=&\frac{1}{
\Gamma(s)}\int_0^1 t^{s-2}e^{-xt}dt+\frac{1}{
\Gamma(s)}\int_0^1 t^{s-1}\left(\frac{1}{1-e^{-t}}-\frac{1}{t} \right)e^{-xt}dt\\
&+\frac{1}{
\Gamma(s)}\int_1^\infty t^{s-1}\frac{e^{-xt}}{1-e^{-t}}dt
\end{split}
\]
It is clear that the first and the third integrals have a analytic continuation 
to $s=0$, and their derivatives with respect to $s$ at $0$ is an analytic function in $x$. 
\[
\frac{1}{1-e^{-t}}-\frac{1}{t}
\]
is an analytic function for any $t$. We conclude that 
$L(x)$ is an analytic function on $(0,+\infty)$, and hence on $\mathrm{Re}(x)>0$. \\

Let $x \in \C$. Let $n_x=[-\mathrm{Re}{(x)}]+1$. The series
\[
\sum_{n\geq n_x}\frac{1}{(n+x)^s}=\sum_{n=0}^\infty\frac{1}{(n+n_x+x)^s},
\]
is well defined.  By definition,
\[
L(n_x+x)= \Rprod_{n=0}^\infty  (n+n_x+x)
\]
We set
\[
 \Rprod_{n=0}^\infty (n+x)=L(n_x+x)\prod_{n=0}^{n_x-1}(n+x).
\]
So, $\ell$ can be continued analytically to $\C$.

 Note that this result is a special case of theory developed in \cite{Voros}.\\

\item

\begin{enumerate}
\item
By definition, we have for any $x>0$,
$L(x)=\exp\left(-\frac{\pt}{\pt s}\zeta_H(s,x)_{|_{s=0}}\right)$. 
 It is clear 
that $\zeta_H(s,x)\in \R$ when $\mathrm{Re}(s)>1$, and its meromorphic continuation is   a real-valued function on $\R$. So,
\[{\frac{\pt}{\pt s}\zeta_H(s,x)_{|_{s=0}}} \in \R\;\text{ for any $x>0$},\]
 and hence
\[
L(x)>0,\quad \forall x>0.
\]

\item 
Since,
\[
\zeta_H(s,x+1)=\zeta_H(s,x)-\frac{1}{x^s},\quad \forall x>0,
\]
Then,
\[
L(x+1)=\frac{1}{x}L(x)\quad \forall x>0.
\]
By the uniqueness of the analytic continuation, the above identity holds on $\C$.

We know that
\[
L(x)L(-x)= \Rprod_{n=0}^\infty  (n^2-x^2).
\]
\item 

By Bohr-Mollerup theorem (see \cite[Theorem 2.1 p.14]{ArtinGamma}),
\[
L(x)=\frac{\prod_{n=1}^\infty n}{ \Gamma(x)}\quad\forall x\in \C,
\]
(we recall that $L(1)=\prod_{n=1}^\infty n$).\\ 

According to \cite{Voros}, $L(z)$ is a holomorphic function of the complex variable $z$ whose zeros (counted with multiplicity) are the numbers $-1,-2,\ldots $, which is by bounded by
$\exp(a+b|z|^N)$  for some constants $a,b$ and $N$ (i.e. $L(z)$ is a function of finite order, see \cite[p. 138]{SteinComplex}).  
The sine function is an entire function of finite order.  \\

We can conclude that the following function \[
f(z)=\frac{z \sin(\pi z)}{L(z)L(-z)},
\]
is entire  of finite order without zeros on $\C$. By Hadamard's factorization theorem, there exists a polynomial $P$ such that
\[
f(z)=e^{P(z)}.
\]
We can verify easily that $f(z+1)=f(z)$   for any $z\in \C$, and since $P$ is continuous,  then $P(z+1)-P(z)$ is the constant polynomial. It follows that  $P(z)=\alpha z+\beta$ for some $\alpha,\beta\in \C$.  
Taking the logarithmic derivative of $f$, we get
\[
a=\frac{f'(z)}{f(z)}=\frac{f'(0)}{f(0)}\quad \forall z\in \C.
\]
But $f(z)=f(-z)$ for any $z\in \C$. Then,
$a=0$, and hence  $f(z)=f(0)$ for any $ z\in \C.$ That is,
\[
 \frac{z \sin(\pi z)}{L(z)L(-z)}=-\frac{\pi }{\Rprod_{n=1}^\infty n^2}\quad\forall z\in \C.
\]
 
 Recall that $L(z)=(\Rprod_{n=1}^\infty n)/\Gamma(z)$. Then, we obtain the reflection formula for the Gamma function:
 \begin{equation}\label{reflection2}
- z \sin(\pi z)=\frac{\pi}{\Gamma(z)\Gamma(-z)}\quad \forall z. 
 \end{equation}
(we have used  that $\Rprod_{n=1}^\infty n^2=(\Rprod_{n=1}^\infty n)^2$).
\end{enumerate}
\end{proof}

 \vskip 1cm
 
 Let us evaluate $\Rprod_{n=1}^\infty n$. We have
\[
\int_1^\infty \frac{1}{x^s}dx=\int_0^1 \zeta_H(s,x+1)dx,\quad \forall \ \mathrm{Re}(s)>1.
\]
By a meromorphic continuation, and elementary computations, we obtain,
\[
-1=\int_0^1\frac{\pt}{\pt s}\zeta_H(0,x+1)dx=-\log( \prod_{n=1}^\infty n) -1+\int_0^1 \log \Gamma(x)dx.
\]
That is,
\[
\int_0^1 \log \Gamma(x)dx=\log( \prod_{n=1}^\infty n).
\]
The computation of $\int_0^1\log \Gamma(x)dx$ is well known, and uses \ref{reflection2}. For reader's convenience, we recall the proof here. We have
\[
\begin{split}
\int_0^1 \log \Gamma(x)dx=&\int_0^1 \log \Gamma(1-x)dx\\
=& \log \pi -\int_0^1 \log \Gamma(x)dx  -\int_0^1 \log \sin (\pi x) dx \quad \text{(by \ref{reflection2})}.
\end{split}
\] 
Using  an elementary change of variables,   $-\int_0^1 \log \sin (\pi x) dx=-\log 2$. We conclude that
\[
\int_0^1 \log \Gamma(x)dx=\frac{1}{2}\log (2\pi).
\]
Then,
\[
\Rprod_{n=1}^\infty n=(2\pi)^\frac{1}{2}. 
\]

We conclude that
\[
L(z)L(-z)= 2 z \sin(\pi z)\quad\forall z\in \C.
\]
This ends the proof of the theorem.
\bibliographystyle{plain}

\bibliography{biblio}

\end{document}